\documentclass[10pt,leqno]{amsart}
\usepackage{indentfirst, amssymb,amsmath,amsthm}
\newtheorem*{theoA}{Theorem A}
\newtheorem*{theoB}{Theorem B}
\newtheorem*{theoC}{Theorem C}
\newtheorem*{theoD}{Theorem D}
\newtheorem*{theoE}{Theorem E}
\newtheorem*{theoF}{Theorem F}

\newtheorem{theo}{Theorem}[section]
\newtheorem{lem}{Lemma}[section]
\newtheorem{cor}{Corollary}[section]

\newtheorem{exm}{Example}[section]
\newtheorem{defi}{Definition}[section]
\newtheorem{rem}{Remark}[section]
\newcommand{\ol}{\overline}
\newcommand{\be}{\begin{equation}}
\newcommand{\ee}{\end{equation}}
\newcommand{\beas}{\begin{eqnarray*}}
\newcommand{\eeas}{\end{eqnarray*}}
\newcommand{\bea}{\begin{eqnarray}}
\newcommand{\eea}{\end{eqnarray}}
\newcommand{\lra}{\longrightarrow}
\numberwithin {equation}{section}
\numberwithin {lem}{section}
\numberwithin {theo}{section}
\numberwithin {defi}{section}
\numberwithin {rem}{section}
\numberwithin {cor}{section}

\begin{document}
	\title[Weighted sharing and uniqueness of $\boldmath{L}$-Function ]{ Weighted sharing and uniqueness of $\boldmath{L}$-Function with certain class of meromorphic function}
		\date{}
	\author{Abhijit Banerjee\;\;\; and \;\;\;Arpita Kundu}
	\date{}
	\address{ Department of Mathematics, University of Kalyani, West Bengal 741235, India.}
	\email{abanerjee\_kal@yahoo.co.in, abanerjeekal@gmail.com}
	
	\address{Department of Mathematics, University of Kalyani, West Bengal 741235, India.}
	\email{arpitakundu.math.ku@gmail.com}
	\maketitle
	\let\thefootnote\relax
	\footnotetext{Type set by \AmS -\LaTeX}
	\begin{abstract}The purpose of the paper is to study the uniqueness problem of a $L$ function in the Selberg class sharing one or two sets with an arbitrary meromorphic function having finite poles. We manipulate the notion of weighted sharing of sets to improve one result of Yuan-Li-Yi [Value distribution of $L$-functions and uniqueness questions of F. Gross, Lithuanian Math. J., $\mathbf{58(2)}$(2018), 249-262]. More importantly, we have pointed out a number of gaps in all the results of  Sahoo-Halder [Results on $L$ functions and certain uniqueness question of Gross, Lithuanian Math. J., $\mathbf{60(1)}$(2020), 80-91] which actually makes the validity of the same paper under question. As an attempt to rectify the results of Sahoo-Halder we have presented the accurate forms and proof of the results in a compact and convenient manner.
		
	\end{abstract}
\section{introduction}	
By a meromorphic function we shall always mean a meromorphic function in the
complex plane. We adopt the standard notations of Nevanilinna theory of meromorphic functions as explained in \cite{W.K.Hayman_64}.
\par  Let $f$ and $g$ be two non-constant meromorphic functions and let $S$ be a subset of distinct elements in the complex plane. For some $a\in\mathbb{C}\cup\{\infty\}$, we define $E_{f}(S)= \displaystyle\cup_{a\in s}\{z:f(z)-a=0\}$, where each point is counted according to its multiplicity. If we do not count the multiplicity then the set $\cup_{a\in S}\{z:f(z)-a=0\}$ is denoted by $\ol E_{f}(S)$. If $E_{f}(S)=E_{g}(S)$ then we say $f$ and $g$ share the $S$ CM. On the other hand, if $\ol E_{f}(S)=\ol E_{g}(S)$ then we say $f$ and $g$ share the $S$ IM. This paper deals with the uniqueness problems of value sharing and set sharing related to $L$-functions and an arbitrary meromorphic function in $\mathbb{C}$.


\par In 1989, Selberg \cite{selberg-92} introduced a new class of Dirichlet series, called the Selberg class, which later became an important field of research in analytic number theory. In this paper, by an $L$-function we mean a Selberg class function with the Riemann zeta function $\zeta(s)=\sum_{n=1}^{\infty}\frac{1}{n^{s}}$
 as the prototype. The
Selberg class $\mathcal{S}$ of $L$-functions is the set of all Dirichlet series $\mathcal{L}(s)=\sum_{n=1}^{\infty}
a(n)n^{-s}$ of a complex variable
$s$ that satisfy the following axioms (see \cite{selberg-92}):
\\$\boldsymbol{(i)}$ Ramanujan hypothesis: $a(n) \ll n^{\epsilon}$ for every $\epsilon > 0$.
\\$\boldsymbol{(ii)}$ Analytic continuation: There is a non-negative integer $k$ such that $(s- 1)^{k}\mathcal{L}(s)$ is an entire function of finite order.
\\$\boldsymbol{(iii)}$ Functional equation: $\mathcal{L}$ satisfies a functional equation of type $$\Lambda _{\mathcal{L}}(s) = \omega\ol {\Lambda _{\mathcal{L}}(1 - \ol s)},$$ where $$\Lambda _{\mathcal{L}}(s)=\mathcal{L}(s)Q^{s}\prod_{j=1}^{K}\Gamma(\lambda _{j}s + \nu_{j})$$ with positive real numbers Q, $\lambda_{j}$ and complex numbers $\nu_{j} , \omega$  with 
$ Re \nu_{j}\geq 0$ and $|\omega|=1$.
\\$\boldsymbol{(iv)}$ Euler product hypothesis : $\mathcal{L}$ can be written over prime as $$\mathcal{L}(s)=\prod_{p}\exp\left(\sum_{k=1}^{\infty}b(p^{k})/p^{ks}\right)$$ with suitable coefficients $b(p^{k})$ satisfying $b(p^{k})\ll p^{k\theta}$ for some $\theta<1/2$ where the product is taken over all prime numbers $p$.
\par The Ramanujan hypothesis implies that the Dirichlet series $\mathcal{L}$ converges absolutely in the half-plane $Re( s) > 1$ and then is extended meromorphically. The degree $d_{\mathcal{L}}$ of an $L$-function $\mathcal{L}$ is defined to be \beas d_{\mathcal{L}}=2\sum_{j=1}^{K}\lambda_{j},\eeas

\par where $\lambda_{j}$ and $K$ are respectively the positive real number and the positive integer as in axiom (iii) above.
For the last couple of years or so, the researchers  have given priority to the investigations on the value distributions of $L$-functions (see \cite{Graun-Grahl-Steuding}, \cite{Hu_LI_Can-16}, \cite{B.Q.LI-Proc.Am-10}, \cite{Li-Yi_Nachr}, \cite{Steuding-Sprin-07}). The value distribution of an $L$-function $\mathcal{L}$ concerns about the roots of the equation $\mathcal{L}(s)=c$ for some $c\in \mathbb{C}\cup\{\infty\}$. In case we talk about the sharing of sets by an $L$-function, we refer the reader to the first paragraph of this paper where all the definitions discussed also applicable to an $L$-function. Regarding uniqueness problem of two $\mathcal{L}$ functions, in 2007, Steuding [p. 152, \cite{Steuding-Sprin-07} ] proved that the number of shared values can be reduced significantly. Below we invoke the result.

\begin{theoA}
	If two $L$-functions $\mathcal{L}_{1}$ and $\mathcal{L}_{2}$ with $a(1)=1$ share a complex value $c\;(\not=\infty)$ CM, then $\mathcal{L}_{1}=\mathcal{L}_{2}$.\end{theoA}
	\begin{rem}
		Providing a counterexample, Hu and Li \cite{Hu_LI_Can-16} have pointed out that {\it{Theorem A}} is not true when c = 1. \end{rem} 
	
	Since $L$-functions possess meromorphic continuations, it is quite natural to investigate up to which extent an $L$-function can share values with an arbitrary meromorphic function. In 2010, Li \cite{B.Q.LI-Proc.Am-10}  observed that {\it{Theorem A}} no longer holds for an $L$-function and a meromorphic function, which is clear from the following example.
	\begin{exm}
		For an entire function g, the functions $\zeta$ and $\zeta e^{g}$ share $0$ CM, but $\zeta \not=\zeta e^{g}$.
	\end{exm}
However, considering two distinct complex values, Li \cite{B.Q.LI-Proc.Am-10}  proved the following uniqueness result.
\begin{theoB}
	Let $f$ be a meromorphic function in $\mathbb{C}$ having finitely many poles, and let a and b be any two distinct finite complex values. If $f$ and a non-constant $L$-function $\mathcal{L}$ share $a$ CM and $b$ IM, then $f=\mathcal{L}$.
\end{theoB}  
\par Next to streamline all the results we are going to demonstrate onward, let us define the two polynomials $P(w)$, $P_{1}(w)$ as follows: $$P(w)=w^{n}+aw^{m}+b\;\; \text {and}\;\;P_{1}(w)=w^{n}+aw^{n-m}+b$$ where $a,\;b\in\mathbb{C}\backslash\{0\}$ and $n$, $m$ are two positive integers such that gcd$(n,m)=1$. \par In view of {\it Lemma {\ref{l2.4}}} proved afterwards, we see that both $P(w)$ and $P_{1}(w)$ can have at most one multiple zero. Next corresponding to the polynomials $P(w)$, $P_{1}(w)$, let us define two sets $S$ and $S_{1}$ as follows: \bea\label{e1.1} S=\{w:P(w)=0\}=\{\alpha_{1},\alpha_{2},\ldots,\alpha_{l}\}\eea and  \bea\label{e1.2} S_{1}=\{w:P_{1}(w)=0\}=\{\beta_{1},\beta_{2},\ldots,\beta_{l}\}\eea where $n-1\leq l\leq n$.
\par Inspired by the famous question of Gross \cite{G.F_Springer(1977)} for meromorphic functions, Yuan-Li-Yi \cite{Yan-Li-Yi_Lith} proposed the question ``what can be said about the relationship
between a meromorphic function $f$ and an $L$-function $\mathcal{L}$ if $f$ and $\mathcal{L}$ share one or two finite sets?''   
	In this respect, to find the relationship between a meromorphic function $f$ and an $L$-function $\mathcal{L}$ sharing one or two finite sets, Yuan-Li-Yi \cite{Yan-Li-Yi_Lith} proved the following uniqueness result.	
\begin{theoC}\cite{Yan-Li-Yi_Lith}
 Let $f$ be a meromorphic function having finitely many poles in $\mathbb{C}$ and let ${\mathcal{L} } $ be a non-constant $L$-function. 	Let $S$ be defined as in (\ref{e1.1}) and $n(\geq 5)>m$ and $c\in\mathbb{C}\setminus S\cup\{0\}$. If $f$ and $\mathcal{L}$ share $S$ CM and $c$ IM, then $f= \mathcal{L}$.
\end{theoC}
\begin{theoD}\cite{Yan-Li-Yi_Lith}
	Let $f$ and $\mathcal{L}$ be a defined as in {\text Theorem C}.
	Let $S$ be defined as in (\ref{e1.1}) where $n>2m+4$. If $f$ and $\mathcal{L}$ share $S =\{w: P(w) = 0\}$ CM, then $f = \mathcal{L}$.
\end{theoD}
\par Sahoo-Halder \cite{Hal-Sahoo(19)_Lith} proved a supplementary result corresponding to {\it{Theorems C}}, {\it{D}} for IM sharing. Sahoo-Halder \cite{Hal-Sahoo(19)_Lith} proved the following result.  
\begin{theoE}\cite{Hal-Sahoo(19)_Lith}
	Let $f$ be a meromorphic function having finitely many poles in $\mathbb{C}$ and $\mathcal{L}$ be a non-constant
	$L$-function. 
	 Also let $S$ and $c$ be defined same as in {\it{Theorem C}} and  $n > 4k + 9$ where $k = n - m\geq 1$. 
	  Then if $f$ and $\mathcal{L}$ share $S$ IM and $c$ IM, then $f = \mathcal{L}$.
\end{theoE}
\begin{theoF}\cite{Hal-Sahoo(19)_Lith}
		Let $f$ be a meromorphic function having finitely many poles in $\mathbb{C}$ and $\mathcal{L}$ be a non-constant $L$-function. Let $S$ be given as in (\ref{e1.1}), where $n>\max\{2m+4,\; 4k+9\}$ such that $k = n - m\geq 1$. If $f$ and $\mathcal{L}$ share $S =\{w: P(w) = 0\}$ IM, then $f = \mathcal{L}$.
\end{theoF}
We now observe some severe errors in \cite{Hal-Sahoo(19)_Lith} as follows:
\begin{rem}
In \cite{Hal-Sahoo(19)_Lith} Sahoo-Halder gave a restriction on choice
 of $m$, whereas if we consider the polynomial for $m=1,2,\ldots$ then putting $m=1,2$ $(k=n-1,n-2)$ we get $n>4n+5$ and $n>4n+1$ respectively, which is a contradiction. So in general the theorem cease to hold for any arbitrary value of $m$ (or $k$). A careful observation shows that $n > 4k + 9$ holds only when $\frac{3n+9}{4}<m<n$ $(1\leq k\leq\frac{n-10}{4})$. 

 
Next in the proof of {\it{Theorem E}}  in \cite{Hal-Sahoo(19)_Lith} [p.10, l.10 from bottom] the authors concluded that if $f$ and $\mathcal{L}$ share $c$ IM then $P(f)$ and $P(\mathcal{L})$ also share $P(c)$ IM. With the help of this argument they finally set up $P(f)=P(\mathcal{L})$ and proved the rest part of the theorem. But this conclusion is true only when (i)$P(f)-P(c)=K_{1}(f-c)^{n}$ and $P(\mathcal{L})-P(c)=K_{2}(\mathcal{L}-c)^{n}$,  for some constant $K_{1},K_{2}$, or  (ii)$P(f)=P(\mathcal{L})$. In the proof of {\it{Theorem E}} clearly both the arguments (i) and  (ii) fail and so in general $P(f)$, $P(\mathcal{L})$ not suppose to share $P(c)$ for any arbitrary $f$ and $\mathcal{L}$. In other words the proof of {\it{Theorem E}} in \cite{Hal-Sahoo(19)_Lith} is not correct and so there is a gap in the proof of {\it{Theorem E}}.\end{rem}
\begin{rem}According to the {\it{Theorem F}}  \cite{Hal-Sahoo(19)_Lith} we have, if $f$ and $\mathcal{L}$ share $S =\{w: P(w) = 0\}$ IM, then for  $n>\max\{2m+4, 4k+9\}$ one can get $f = \mathcal{L}$, where $n,m,k$ and $f,\mathcal{L}$ are mentioned in {\it{Theorem F}}. If possible let us assume for some $m$ {\it{Theorem F}} holds. Then obviously from given condition we have  $$ n > 4k+9\implies n>4n-4m+9\implies 4m>3n+9$$ and $n>2m+4$, both together implies $4m>6m+21$, which is absurd. So there exist no such $m$ for which  {\it{Theorem F}} is true. Hence validity of {\it{Theorem F}} is also at stake. 
\end{rem}
In view of {\it Remarks 1.2} and {\it 1.3} we see that the very existence of the whole paper \cite{Hal-Sahoo(19)_Lith} is at stake.
 \par In this paper we have improved {\it{Theorem D}} by relaxing the nature of sharing the set with the notion of weighted sharing. We have also presented and proved the corrected form of {\it Theorems E} and {\it F} on the uniqueness of $L$-function and meromorphic functions.
Thus, Sahoo-Halder's  \cite{Hal-Sahoo(19)_Lith} results have been fully rectified. 
 \par Before presenting the main results we invoke the definition of weighted sharing.
 \begin{defi}\cite{Lahiri_Complex.Var(01)}
 	Let  $k$ be a non-negative integer or infinity. For $a\in\mathbb{C}\cup\{\infty\}$ we denote by $E_{k}(a;f)$ the set of all $a$-points of $f$, where an $a$-points of $f$ of multiplicity $m\;(\leq k)$ counted $m$ times and $m(> k)$ then it counted $k+1$ times. If  $E_{k}(a;f)= E_{k}(a;g)$, we say that $f,\;g$ share the value $a$ with weight $k$.
 	\par We write $f,\;g$ share $(a,k)$ to mean that $f,\;g$ share the value $a$ with weight $k$. Clearly if $f,\;g$ share $(a,k)$ then $f,\;g$ share $(a,p)$ for any integer $p$, $0\leq p < k$. Also note $f,\;g$ share a value $a$ IM or CM $f,\;g$ share $(a,0)$ or $(a,\infty)$ respectively.   
 \end{defi}
 \begin{defi}\cite{Lahiri_Archivum}
 	For $ S\subset \mathbb{C}\cup\{\infty\}$, we define $E_{f}(S,k)=\cup_{a\in S}E_{k}(a;f)$, where $k$ is a non-negative integer or infinity. Clearly $E_{f}(S)=E_{f}(S,\infty)$.\par In particular $E_{f}(S,k)=E_{g}(S,k)$ and $E_{f}(\{a\},k)=E_{g}(\{a\},k)$ implies $f$ and $g$ share the set $S$ and the value $a$ with weight $k$ .
 	 
 \end{defi}
  We first present the following theorem corresponding to {\it Theorem D} which provide the corrected form of {\it Theorem F} as well.

\begin{theo}\label{t1.1}
 Let $S$ be defined as in (\ref{e1.1}). Also let $f$ be a meromorphic function having finitely many poles in $\mathbb{C}$ and $\mathcal{L}$ be a non-constant $L$-function such that $E_{f}(S,s)=E_{\mathcal{L}}(S,s)$. If
\\(i) $s\geq 2$ and $n>2m+4$, or if
\\(ii) $s=1$ and $n>2m+5$, or if
\\(iii) $s=0$ and $n>2m+10$;
 then $f = \mathcal{L}$.
\end{theo}
The following corollary is obvious from the above theorem which also relax the CM sharing of $S$ in {\it{Theorem D}} to weight $2$.
\begin{cor}\label{cor1.2}
	Let $S$ be defined as in (\ref{e1.1}) and $n > 2m+4$. Also let
	$f$ be a meromorphic function having finitely many poles in $\mathbb{C}$ and let ${\mathcal{L} } $ be a non-constant $L$-function. If $E_{f}(S,2)=E_{\mathcal{L}}(S,2)$, then $f= \mathcal{L}$.
\end{cor}
We note that as far as the set $S$ is concerned, {\it Theorem \ref{t1.1}} is not valid for all $m$ with gcd$(n,m)=1$. For example if we consider $m=n-1$ for $s\geq 2$, {\it Theorem \ref{t1.1}} is not applicable. An elementary calculation will show that when $1\leq m <\frac{n-4}{2}$, for $s\geq 2$, the theorem is valid. Similarly it can be shown that for $s=1$ or $0$ then $m$ will have some restrictions. So it will be interesting to investigate the form of {\it Theorem \ref{t1.1}} for the rest values of $m$ in order to complete the theorem. The following theorem elucidate in this regard.     
\begin{theo}\label{t1.2}
	Let	$S_{1}$ be defined as in (\ref{e1.2}), $f$ be a meromorphic function having finitely many poles in $\mathbb{C}$ and $\mathcal{L}$ be a non-constant $L$-function such that $E_{f}(S_{1},s)=E_{\mathcal{L}}(S_{1},s)$. If
	\\(i) $s\geq 2$ and $n>2m+4$, or if
	\\(ii) $s=1$ and $n>2m+5$, or if
	\\(iii) $s=0$ and $n>2m+10$;
then $f = \mathcal{L}$.
\end{theo}
\begin{cor}
In {\it Theorem \ref{t1.2}} we see that $\frac{n+4}{2}<k=n-m\leq n-1$ and a close look will reveal that this will supplement the values of $m$ for the case $s\geq 2$ in {\it Theorem \ref{t1.1}}. 
\end{cor}
In {\it{Theorem C}}, $c$ has been considered as non-zero, so it will be interesting to investigate the theorem for $c=0$. In the next theorem, we have rectified {\it{Theorem E}} considering two special form of $c$. However we have not succeeded to get the result for any arbitrary $c$.
\begin{theo}\label{t1.3}
	Let $S$ be defined as in (\ref{e1.1})  and $f$ be a meromorphic function having finitely many poles in $\mathbb{C}$ and let $\mathcal{L}$ be a non-constant
	$L$-function. 
	Suppose $E_{f}(S,s)=E_{\mathcal{L}}(S,s)$ and $E_{f}(\{c\},t)=E_{\mathcal{L}}(\{c\},t)$ for some finite $c\in( 0, a_{1},a_{2},\ldots,a_{n-m})$ but $c\not\in S$, where $a_{i}\;(i=1,2,\ldots,n-m)$ are zeros of $nz^{n-m}+ma$. First suppose \\(I) $c=0$, $t=0$ and  
		\\(i) $s\geq 2$, $n\geq2m+3$ or
	 	\\(ii) $s=1$, $n\geq2m+4$ or 
		\\(iii) $s=0$, $n\geq2m+9$; then  
				we have $f = \mathcal{L}$.\par Next suppose \\(II) $c$ is a root of $nz^{k}+ma=0$, $k=n-m$ and $l=n$. If 
				\\(i) $s\geq 2$, $t=1$ and $n\geq 2k+3$ or 
		\\(ii) $s=1$, $t=0$ and $n\geq 2k+4$ or \\(iii) $s=0$, $t=0$ and $n\geq 2k+7$, then  
						we have $f = \mathcal{L}$. 
\end{theo}

\par We assume that the readers are familiar with the standard notations of Nevanlinna theory such as the Nevanlinna characteristic function $T(r,f)$, the proximity function $ m(r,f)$, the reduced counting function $\ol N(r,\infty;f)$, and so on, which are well explained in \cite{W.K.Hayman_64}. Here we use the symbol $\rho(f)$ to denote the order of a non-constant meromorphic
function $f$, which is defined as $$\rho(f)=\limsup_{r\lra\infty}\frac{\log ^{+}T(r,f)}{\log r}$$

\par By $S(r,f)$ we mean any quantity satisfying $S(r,f)=O(\log (rT(r,f)))$, for all r possibly outside a set of finite Lebesgue measure. If $f$ is a function of finite order, then $S(r,f)=O(\log r)$ for all r. In this paper we consider $f$ as a non-constant meromorphic function having finitely many poles in $\mathbb{C}$, then clearly $\ol N(r,\infty;f)=O(\log r)$.
We now explain some more notations and definitions which are used in this paper.

\begin{defi}\cite{Lahiri_Complex.Var(01)}
	For $a\in \mathbb{C}\cup\{\infty\}$ we denote by $N(r,a;f\mid=1)$ the counting function of simple $a$ points of $f$. For a positive  integer $s$ we denote by $N(r,a;f\mid\leq s)\;(N(r,a;f\mid\geq s))$ the counting function of those $a$-points of $f$ whose multiplicity are not greater(less) than $s$, where each $a$-point is counted according to it's multiplicity.
	\par   Also $ \ol N(r,a;f\mid\leq s)\;(\ol N(r,a;f\mid\geq s))$ is defined similarly, where in counting the $a$-points counted exactly once. 
\end{defi}
\begin{defi}\cite{Lahiri_Archivum}
	We denote $N_{2}(r,a;f)=\ol N(r,a;f)+\ol N(r,a;f\mid\geq 2)$.
\end{defi}
\begin{defi}\cite{Lahiri_Archivum}
	If $s$ is a positive integer, we denote by $N(r,a;f\mid=s)$ the counting function of those a points of $f$ whose multiplicity is $s$, where each point counted according to its multiplicity. \par   Also $ \ol N(r,a;f\mid=s)$ is defined similarly, where in counting the $a$-points counted exactly once. 
	\par Let $z_{0}$ be $a$ point of $f$ and $g$ of multiplicity $p$ and $q$ respectively. Then by $N^{1)}_{E}(r,a;f)$ we denote the counting function of those a-points of $f$ and $g$ where $p = q = 1$. 
\end{defi}
\begin{defi}\cite{Lahiri_Archivum}
	Let $f,\;g$ share a value $a$ IM. We denote by $\ol N_{*}(r,a;f,g)$ the counting function of those $a$-points of $f$ whose multiplicities are different from multiplicities of the corresponding $a$-points of $g$, where each $a$-points is counted exactly once. \par Clearly $\ol N_{*}(r,a;f,g)=\ol N_{*}(r,a;g,f)=\ol N_{L}(r,a;f)+\ol N_{L}(r,a;g)$.  
\end{defi}
\begin{defi}\cite{Ban_Tamkang}
	Let $a,b_{1},b_{2},\ldots,b_{q}\in\mathbb{C}\cup\{\infty\}$. We denote by $N(r,a;f\mid g\not=b_{1},b_{2},\ldots,b_{q})$ the counting function of those $a$-points of $f$, counted according to its multiplicity, which are not $b_{i}$ points of $g$ for $i=1,2,\ldots,q$.
\end{defi}
\begin{defi}\cite{Ban-Lahiri_Comput(2012)}
	Let $P(z)$ be a polynomial such that $P^{'}(z)$ has mutually $k$ distinct  zeros given by $d_{1}, d_{2}, \ldots, d_{k}$ with multiplicities $q_{1}, q_{2}, \ldots, q_{k}$ respectively. Then $P(z)$ is said to satisfy the critical injection property if $P(d_i)\not =P(d_j)$ for $i\not=j$, where $i,j\in \{1,2,\cdot\cdot\cdot,k\}$.
\end{defi}
\section{lemma}
Next, we present some lemmas that will be needed in the sequel. Henceforth, we denote by $H$, $\Phi$ the following functions :
$$H=\bigg(\frac{F''}{F'}-\frac{2F'}{F-1}\bigg)-\bigg(\frac{G''}{G'}-\frac{2G'}{G-1}\bigg)\;$$
and $$\Phi=\frac{F'}{F-1}-\frac{G'}{G-1}$$
\begin{lem}\label{l2.1}\cite{Yi-kodai(1999)}
	Let $F$ and $G$ share $(1,0)$ and $H\not\equiv 0$. Then,
	$$N^{1)}_{E}(r,1;F)\leq N(r,H)+S(r,F)+S(r,G).$$
\end{lem}
\begin{lem}\label{l2.2}\cite{Lahiri_Complex.Var(01)}
	If two non-constant meromorphic function $F$ and $G$ share $(1,m)\;$ and $H\not\equiv 0$ then, \beas N(r,\infty;H)&\leq& \ol N(r,0;F\mid \geq 2)+\ol N(r,0;G\mid\geq 2)+\ol N_{*}(r,1;F,G)+\ol N(r,\infty;F\mid\geq 2)+\ol N(r,\infty;G\mid\geq 2)\\&&+\ol N_{0}(r,0;F')+\ol N_{0}(r,0;G'),\eeas where, $\ol N_{0}(r,0;F')$ is the reduced counting function for those zeros of $F'$, which are not zeros of  $F(F-1)$
	and $\ol N_{0}(r,0;G')$ is similarly defined.\end{lem}
\begin{lem}\label{l2.3}\cite{Mokhon'ko-lem}
	Let $P(f)=\sum_{k=0}^{n} a_{k}f^{k}/ \sum_{j=0}^{m}b_{j}f^{j},$ be an irreducible polynomial in $f$, with constants coefficient $ \{a_{k}\}$ and $\{b_{j}\}$ where $a_{n}\not=0$ and $b_{m}\not=0$. Then $$T(r,P(f))=dT(r,f)+S(r,f),$$ where $d=\max\{m,n\}.$
\end{lem}
\begin{lem}\label{l2.4}
Let	$P(w) =w^{n} + aw^{m} + b $ and $P_{1}(w) =w^{n} + aw^{n-m} + b $ be two polynomials,  where $n$ and $m$ be relatively prime positive integers and	$a$, $b$ be two non-zero constants. Then they are critically injective polynomials with at most one multiple zero and at least $n-2$ simple zeros, where the multiplicity of the multiple zero is exactly $2$.
\end{lem}
\begin{proof}

Clearly from the given polynomial $P(w)$, we have $P'(w)=z^{m-1}(nz^{n-m}+ma)$. Suppose the zeros of $P'(w)$ are $0,a_{1},a_{2},\ldots,a_{n-m}$, where $a_{i}^{,}s$ $(i=1,2,\ldots,n-m)$ are given as in {\it{Theorem 1.3}}.\par Clearly $P(0)\not=P(a_{i})$ for $i=1,2,\ldots,n-m$. Now suppose contrary to the statement of the lemma, $a_{i},\;a_{j}\;(1\leq i\not= j\leq n-m)$, $P(a_{i})=P(a_{j})$.
\par Now since $a_{i}$, $a_{j}$ are the zeros of $nz^{n-m}+ma$, we have $na_{i}^{n-m}+ma= na_{j}^{n-m}+ma\implies a_{i}^{n-m}=a_{j}^{n-m},$ using this, from $P(a_{i})=P(a_{j})$ we get $a_{i}^{m}=a_{j}^{m}$. Since gcd$(n,m)=1$, then from $a_{i}^{n-m}=a_{j}^{n-m}$ and $a_{i}^{m}=a_{j}^{m}$, we get $a_{i}=a_{j}$, a contradiction. Hence $P(a_{i})=P(a_{j})\implies a_{i}=a_{j}$ or in other words $P(a_{i})\not=P(a_{j})$ for $i\not=j$. So from the definition it follows that $P(w)$ is critically injective. Also it is obvious from the definition of critically injective polynomial, it has at most one multiple zero. For, if not then the two distinct multiple zeros say $\zeta_1$ and $\zeta_2$ would yield $P(\zeta_1)=P(\zeta_2)$, where $\zeta_1, \zeta_2 \in \{a_{1},a_{2},\ldots,a_{n-m}\}$. Therefore $P(w)$ must contain at most one multiple zero of multiplicity two. \par Similarly one can prove the result for $P_{1}(w)$ and hence the proof is complete.
\end{proof}

\begin{lem}\label{l2.5}\cite{Ban_Tamkang}
	If $F$ and $G$ share $(1,s)$ then $$\ol N(r,1;F)+\ol N(r,1;G)+\left(s-\frac{1}{2}\right)\ol N_{*}(r,1;F,G)-N^{1)}(r,1;F) \leq \frac{1}{2}\big( N(r,1;F)+ N(r,1;G)\big).$$
\end{lem}
\begin{lem}\label{l2.6}
		Let $F=-\frac{f^{n}}{af^{m}+b}$ and $G=-\frac{g^{n}}{ag^{m}+b}$, where $f$ and $g$ are any two non-constant meromorphic functions 
		and n, m are relatively prime positive integers , such that $n > m \geq 1$, and a, b are non-zero finite constants. Let $\gamma_{i},\;i=1,2,\ldots,m$ are roots of $aw^{m}+b=0$. If $H\not=0$ and $E_{f}(S,s)=E_{g}(S,s)$ where $S$ is defined as in (\ref{e1.1}) then, 
		\beas\frac{n}{2}\big(T(r,f)+T(r,g)\big)&\leq& 2\big(\ol N(r,0;f)+\ol N(r,0;g)+\ol N(r,\infty;f)+\ol N(r,\infty;g)\big)+\sum_{i=1}^{m}\big(N_{2}(r,\gamma_{i};f)\\&&+N_{2}(r,\gamma_{i};g)\big)+\left(\frac{3}{2}-s\right)\ol N_{*}(r,1;F,G)+S(r,f)+S(r,g).\eeas
\end{lem}
\begin{proof}
	Clearly here $F$ and $G$ share $(1,s)$. 
	By the Second Fundamental Theorem and using {\it{Lemmas \ref{l2.3}}}, {\it{\ref{l2.1}}}, {\it{\ref{l2.2}}} and {\it{\ref{l2.5}}} we have, \beas && n(T(r,f)+T(r,g))= T(r,F)+T(r,G)\\&\leq& \ol N(r,1;F)+\ol N(r,0;F)+\ol N(r,\infty;F) +\ol N(r,1;G)+\ol N(r,0;G)+\ol N(r,\infty;G)\\&&-N_{0}(r,0;F')-N_{0}(r,0;G')+S(r,F)+S(r,G) \\&\leq& \frac{n}{2}(T(r,f)+T(r,g))+N^{1)}(r,1,F;G)+\left(\frac{1}{2}-s\right)\ol N_{*}(r,1;F,G)+\ol N(r,0;f)\\&&+\ol N(r,\infty;f)+\ol N(r,0;af^{m}+b)+\ol N(r,0;g)+\ol N(r,\infty;g)+\ol N(r,0;ag^{m}+b)\\&&-N_{0}(r,0;F')-N_{0}(r,0;G')+S(r,f)+S(r,g).\eeas
	i.e., \beas&&\frac{n}{2}(T(r,f)+T(r,g))\\&\leq&2\ol N(r,\infty;f)+2\ol N(r,\infty;g)+2\ol N(r,0;f)+2\ol N(r,0;g) +N_{2}(r,0;af^{m}+b)\\&&+N_{2}(r,0;ag^{m}+b)+\left(\frac{3}{2}-s\right)\ol N_{*}(r,1;F,G)+S(r,f)+S(r,g)\\&\leq&2\big(\ol N(r,0;f)+\ol N(r,0;g)+\ol N(r,\infty;f)+\ol N(r,\infty;g)\big)+\sum_{i=1}^{m}\big(N_{2}(r,\gamma_{i};f)+N_{2}(r,\gamma_{i};g)\big)\\&&+\left(\frac{3}{2}-s\right)\ol N_{*}(r,1;F,G)+S(r,f)+S(r,g). \eeas
\end{proof}

\begin{lem}\label{l2.7}
	Let $F = -(f^{n} + af^{n-m})/b$ and $G = -(g^{n} + ag^{n-m})/b$, where $f$ and $g$ be any two non-constant meromorphic functions 
and	$n$, $m$ be relatively prime positive integers such that $n > m \geq 1$, and $a$, $b$ be non-zero finite constants. Let $\delta_{i},\;i=1,2,\ldots,m$ be the distinct roots of the equation $w^{m}+a = 0$. If $H \not\equiv 0$ and $E_{f}(S_{1},s)=E_{g}(S_{1},s)$ where $S_{1}$ is defined as in (\ref{e1.2}). Then,  \beas\frac{n}{2}\big(T(r,f)+T(r,g)\big)&\leq& 2\big(\ol N(r,0;f)+\ol N(r,0;g)+\ol N(r,\infty;f)+\ol N(r,\infty;g)\big)+\sum_{i=1}^{m}\big(N_{2}(r,\delta_{i};f)\\&&+N_{2}(r,\delta_{i};g)\big)+\left(\frac{3}{2}-s\right)\ol N_{*}(r,1;F,G)+S(r,f)+S(r,g).\eeas\end{lem}
\begin{proof}
	We omit the proof since this is similar to the proof of {\it{Lemma 2.6}}.
\end{proof}
Now before discussing the next lemmas we define some notations\;. \\By $\Theta(Q(w))$  we denote the number of distinct zeros of any polynomial $Q(w)$ of degree $d$ and here \beas \chi^{Q}_{d}&=&0\;, \;\;when\;\Theta(Q(w))=d\\ &=&1\;,\;\;when\;\Theta(Q(w))=d-1.\eeas
\begin{lem}\label{l2.8}
Let $F$ and $G$ be defined as in {\it{Lemma \ref{l2.6}}} and share $(1,s)$ then $$ \ol N_{L}(r,1;F)\leq \frac{1}{s+1}\big(\ol N(r,0;f)+\ol N(r,\infty;f)+\chi^{P}_{n}N(r,\alpha_{j};f)\big)+S(r,f),$$ where  $\alpha_{j}\;(1\leq j\leq n-1)$ be the multiple root 
of the equation $P(w)=w^{n}+aw^{m}+b=0$. 
\end{lem}
\begin{proof} In view of {\it Lemma 2.4}, first suppose that the equation $P(w)=w^{n}+aw^{m}+b=0$ has one multiple root and assume that as $\alpha_{j}$ ($1\leq j\leq n-1$). Then $\chi^{P}_{n}=1$.
	\beas \ol N_{L}(r,1;F)&\leq& \ol N(r,1;F\mid\geq s+2)\\&\leq& \ol N(r,0;F'\mid\geq s+1;F=1)\\&\leq&\frac{1}{s+1}N(r,0;F'\mid\geq s+1;F=1)\\&\leq&\frac{1}{s+1}\big(N(r,0;f'\mid f\not=0) + N(r,\alpha_{j};f)-N_{o}(r,0;f') \big)\\&\leq&\frac{1}{s+1}\big( N(r,0;\frac{f'}{f})+N(r,\alpha_{j};f)-N_{o}(r,0;f') \big)\\&\leq&\frac{1}{s+1}\big(\ol N(r,\infty;f)+\ol N(r,0;f)+N(r,\alpha_{j};f)-N_{o}(r,0;f') \big)+S(r,f)\\&\leq&\frac{1}{s+1}\big(\ol N(r,0;f)+\ol N(r,\infty;f)+N(r,\alpha_{j};f)-N_{o}(r,0;f') \big)+S(r,f), \eeas  where $N_{o}(r,0;f')=N(r,0;f'\mid f\not=0,\alpha_{1},\alpha_{2},\ldots,\alpha_{n-1})$.
	 \par Next suppose all the $n$ roots of $P(w)$ are distinct. Then $\chi^{P}_{n}=0$, and also from above calculations we have $$\ol N_{L}(r,1;F)\leq \frac{1}{s+1}\big(\ol N(r,0;f)+\ol N(r,\infty;f)-N(r,0;f'\mid f\not=0,\alpha_{1},\alpha_{2},\ldots,\alpha_{n})\big)+S(r,f),$$  hence the proof is complete.
\end{proof}
\begin{lem}\label{l2.9}
	Let $F$ and $G$ be defined in {\it{Lemma \ref{l2.7}}}, share $(1,s)$ then $$ \ol N_{L}(r,1;F)\leq \frac{1}{s+1}\big(\ol N(r,0;f)+\ol N(r,\infty;f)+\chi^{P_{1}}_{n}N(r,\beta_{j};f)\big)+S(r,f),$$ where  $\beta_{j}\;(1\leq j\leq n-1)$ be the multiple root 
	of the equation $P_{1}(w)=w^{n}+aw^{n-m}+b=0$. 	
\end{lem}
\begin{proof}
	We omit this proof because it is similar to the proof of {\it{Lemma \ref{l2.8}}}. 
\end{proof}
\begin{lem}\label{l2.10}
Let $f$ be a non-constant meromorphic function having finitely many poles and $\mathcal{L}$ be a $L$-function such that $E_{f}(S,s)=E_{\mathcal{L}}(S,s)$, where $S$ is defined as in (\ref{e1.1}). Then for $n>2$, $$\frac{f^{n}}{af^{m}+b}.\frac{\mathcal{L}^{n}}{a\mathcal{L}^{m}+b}\not=1.$$
\end{lem}
\begin{proof}
	Contrary to the hypothesis let us take  
	\ $$\frac{f^{n}}{af^{m}+b}.\frac{\mathcal{L}^{n}}{a\mathcal{L}^{m}+b}=1.$$ By {\it{Lemma \ref{l2.3}}} we have \bea\label{e2.1} T(r,f)=T(r,\mathcal{L})+S(r,\mathcal{L}).\eea
	\par Since $f$ has finite number of poles and $\mathcal{L}$ has at most one pole, we have $\ol N(r,\infty;f)=\ol N(r,\infty;\mathcal{L})=O(\log r)$. Also from {\it{Lemma \ref{l2.11}}} it is obvious that $S(r,f)=S(r,\mathcal{L})=O(\log r)$. \par As mentioned in {\it Lemma \ref{l2.6}} we know that $\gamma_{1},\gamma_{2},\ldots,\gamma_{m}$ are roots of $aw^{m}+b=0$. Let $z_{0}$ be a zero of $af^{m}+b$ with multiplicity $p$ and also a zero of $\mathcal{L}$ with multiplicity $q$. \par Then $p=nq\implies p\geq n$. \par Thus $$\ol N(r,\gamma_{i};f)\leq \frac{1}{n}N(r,\gamma_{i};f).$$ By the same arguments one can show $$\ol N(r,\gamma_{i};\mathcal{L})\leq \frac{1}{n}N(r,\gamma_{i};\mathcal{L})$$  Now using this, (\ref{e2.1}) and the Second Fundamental Theorem we get
	\beas mT(r,f)&\leq& \ol N(r,0;f)+\ol N(r,\infty;f)+\sum_{i=1}^{m}\ol N(r,\gamma_{i};f)+S(r,f)\\&\leq&\ol N(r,\infty;\mathcal{L})+\ol N(r,0;a\mathcal{L}^{m}+b)+\frac{1}{n}\sum_{i=1}^{m} N(r,\gamma_{i};f)+O(\log r)\\&\leq& \frac{2m}{n}T(r,f)+O(\log r),\eeas which gives a contradiction.
	 
\end{proof}
\begin{lem}\label{l2.11}
	Let $f$ be a meromorphic function having finitely many poles in $\mathbb{C}$ and $S$ (or $S_1$) be defined as in (\ref{e1.1}) ((\ref{e1.2})). If $f$ and a non-constant $L$-function $\mathcal{L}$ share the set $S\;(or\;S_{1})$ IM, then $\rho(f) =\rho(\mathcal{L}) = 1.$
\end{lem}
\begin{proof}
We omit the proof as it can be found out in the proof of Theorem 5, \{p. 6, \cite{Yan-Li-Yi_Lith}\}.
\end{proof}
\begin{lem}\label{l2.12}
	Let $F$ and $G$ be defined in {\it{Lemma \ref{l2.6}}} such that they share $(1,s)$ and $f$, $g$ share $\{0\}$ IM, then \beas \ol N(r,0;f)=\ol N(r,0;g)&\leq& \frac{1}{n-1}\{\ol N_{*}(r,1;F,G)+\ol N(r,\infty;f)+\ol N(r,\infty;g)+\ol N(r,0;af^{m}+b)\\&&+\ol N(r,0;ag^{m}+b)\}+S(r,f)+S(r,g).
	\eeas
\end{lem}
\begin{proof}
	Since $f$, $g$ share $(0,0)$, it follows that \beas\ol N(r,0;f)=\ol N(r,0;g) &\leq& \frac{1}{n-1}N(r,0;\Phi)\\&\leq& \frac{1}{n-1}T(r,\Phi)+O(1)\\&\leq& \frac{1}{n-1}N(r,\infty;\Phi)+S(r,F)+S(r,G)\\&\leq& \frac{1}{n-1}\big(\ol N_{*}(r,1;F,G)+\ol N(r,\infty;F)+\ol N(r,\infty;G)\big)+S(r,f)+S(r,g)\\&\leq& \frac{1}{n-1}\{\ol N_{*}(r,1;F,G)+\ol N(r,\infty;f)+\ol N(r,\infty;g)+\ol N(r,0;af^{m}+b)\\&&+\ol N(r,0;ag^{m}+b)\}+S(r,f)+S(r,g).\eeas
\end{proof}
\section{proofs of the theorems}
\begin{proof}[Proof of Theorem \ref{t1.2}]
	Let us consider $F=-(f^{n} + af^{n-m})/b$ and $G=-(\mathcal{L}^{n} + a\mathcal{L}^{n-m})/b$, where $f$ has finitely many poles. As $f$ and $\mathcal{L}$ share $S_{1}=\{w:P_{1}(w)=0\}$ with weight $s$, then clearly $F$ and $G$ share $(1,s)$. \par It is given that $f$ has finite number of poles, therefore $\ol N(r,\infty;f)=O(\log r)$. As $\mathcal{L}$ has at most one pole, $\ol N(r,\infty;\mathcal{L})=O(\log r)$. 
	Also from {\it{Lemma \ref{l2.11}}} we have $\rho(f)=\rho(\mathcal{L})=1$. Therefore it is obvious that, $S(r,f)=S(r,\mathcal{L})=O(\log r)$. \par Let us consider $H\not\equiv0$. 
	\\{\it{Case I :}} Suppose $w^{n}+aw^{n-m}+b=0$ has no multiple roots. Then using {\it{Lemma \ref{l2.7}}},  and {\it{Lemma \ref{l2.9}}} we have \bea\label{e3.1} &&\frac{n}{2}T(r)\\\nonumber&\leq& 2\big(\ol N(r,0;f)+\ol N(r,0;\mathcal{L})\big)+\sum_{i=1}^{m}\big(N_{2}(r,\delta_{i};f)+N_{2}(r,\delta_{i};\mathcal{L})\big)+\left(\frac{3}{2}-s\right)\ol N_{*}(r,1;F,G)\\\nonumber&&+O(\log r)\\\nonumber &\leq& (2+m)T(r)+\left(\frac{3}{2}-s\right)(\ol N_{L}(r,1;F)+\ol N_{L}(r,1;G)) +O(\log r)\\\nonumber&\leq& (2+m)T(r)+\frac{\left(\frac{3}{2}-s\right)}{s+1}(\ol N(r,0;f)+\ol N(r,0;\mathcal{L}))+O(\log r),
	\eea
	where $T(r)=T(r,f)+T(r,\mathcal{L})$.
	\par Clearly when 
	\beas&(i)& s\geq2,\; n> 2m+4 \;\;or\;\; when\\&(ii)&  s=1,\;\; n> 2m+5\;\; or\;\; when \\&(iii)&  s=0,\;\; n> 2m+10;\eeas  from (\ref{e3.1}) we get a contradiction.
	\\{\it{Case II :}}
	Again considering consider $w^{n}+aw^{n-m}+b=0$ has a multiple root (say $\beta_{j}$), i.e. the equation has $n-1$ distinct  roots. Then proceeding same as in above (\ref{e3.1}) we get \bea\label{e3.2} &&\frac{n}{2}T(r)\\\nonumber&\leq&(2+m)T(r)+\frac{\big(\frac{3}{2}-s\big)}{s+1}\big(\ol N(r,0;f)+\ol N(r,0;\mathcal{L})+N(r,\beta_{j};f)+ N(r,\beta_{j};\mathcal{L})\big)+O(\log r).
	 \eea 
	 	Clearly from (\ref{e3.2}) and in view of the conditions (i), (ii), (iii) given in {\it{Case I}}, 
	 	we arrive at a contradiction again.	
	  \par Therefore $H\equiv 0$ and so integrating both sides we get, \bea\label{e3.3}\frac{1}{G-1}=\frac{A}{F-1}+B,\eea where $A,\;B$ are two constants, $A\not=0$. From {\it{Lemma \ref{l2.3}}} and (\ref{e3.3}) we have, \bea\label{e3.4} T(r,\mathcal{L})=T(r,f)+O(1).\eea At first let $B$ is non-zero. Then $$G-1=\frac{F-1}{A+B(F-1)}.$$ \par If $A-B\not=0$ then zeros of $F+(A-B)/B$ are poles of $G-1=-(\mathcal{L}^{n}+a\mathcal{L}^{n-m}+b)/b$. Now $F+(A-B)/B$ has at least $n-1$ zeros for all $1\leq m< n$. Then $-(\mathcal{L}^{n}+a\mathcal{L}^{n-m}+b)/b $ has more than one pole implies $\mathcal{L}$ has more than one pole, which is a contradiction. \par Therefore $A-B=0$. Then $$G-1=\frac{F-1}{BF}.$$ \par Using the same arguments for  $ m> 2$, one can show $B=0$. But the argument fails if for $m\leq2$, $f$ has two exceptional values (say-$\xi_{1},\xi_{2}$) or for $m=1$, the same has one exceptional value (say-$v\in\{0,-a\}$). Then using the Second Fundamental Theorem we will get
	  a contradiction again. Hence in any case $B=0$.
	   From (\ref{e3.3}) we have \beas G-1&=&\frac{1}{A}(F-1).\\i.e.,\; (\mathcal{L}^{n} + a\mathcal{L}^{n-m}+b)&=&\frac{1}{A}(f^{n} + af^{n-m}+b).\\i.e.,\;  \mathcal{L}^{n} + a\mathcal{L}^{n-m}+b-\frac{b}{A}&=&\frac{1}{A}(f^{n} + af^{n-m}).\eeas Let us consider $A\not=1$ then $b/A\not=b$. So $w^{n}+aw^{m}+b-b/A=0$ has atleast $n-1$ distinct roots $(say\;p_{i},i=1,2,\ldots,n-1)$, then  by the Second Fundamental Theorem, and (\ref{e3.4})  we get from above \beas (n-2)T(r,\mathcal{L})&\leq& \sum_{i=1}^{n-1}\ol N(r,p_{i};\mathcal{L})+\ol N(r,\infty;\mathcal{L})+S(r,\mathcal{L})\\&\leq&\ol N(r,0;f)+\sum_{i=1}^{m}\ol N(r,\delta_{i};f)+O(\log r)\\&\leq&(m+1)T(r,f)+O(\log r), \eeas which implies $n\leq m+3$, a contradiction for $n>2m+4$. \par Therefore our assumption is wrong. Hence $A=1$ and therefore $$F=G.$$ Now, \beas f^{n}-\mathcal{L}^{n}&=&-a(f^{n-m}-\mathcal{L}^{n-m}).\\i.e., \;\mathcal{L}^{m}(h^{n}-1)&=&-a(h^{n-m}-1).\\ i.e., \;\mathcal{L}^{m}&=&-a\frac{h^{n-m}-1}{h^{n}-1}.\eeas At first assume $h\;(=f/\mathcal{L})$ is a non-constant meromorphic function. Then since gcd(m,n)=1, for $n>2m+4$, from above we get $\mathcal{L}$ has more than one pole, which is a contradiction. \par Therefore $h$ is a constant, satisfying $h^{n}-1=h^{n-m}-1=0\implies h=1\implies f=\mathcal{L}$.
	
\end{proof}
\begin{proof}[Proof of Theorem \ref{t1.1}]
	Let us consider $F=-\frac{f^{n}}{af^{m}+b}$ and $G=-\frac{\mathcal{L}^{n}}{a\mathcal{L}^{m}+b}$ , where $f$ has finitely many poles. As $f$ and $\mathcal{L}$ share $S=\{w:P(w)=0\}$ with weight $s$, i.e. $E_{f}(S,s)=E_{g}(S,s)$  then clearly $F$ and $G$ share $(1,s)$. Let us consider $H\not\equiv0$.
	\\{\it{Case I :}} Let us consider $w^{n}+aw^{m}+b=0$ has no multiple root. Then using {\it{Lemma \ref{l2.6}}} and {\it{Lemma \ref{l2.8}}} we have \bea\label{e3.5}\frac{n}{2}T(r)&\leq& 2\big(\ol N(r,0;f)+\ol N(r,0;\mathcal{L})\big)+\sum_{i=1}^{m}\big(N_{2}(r,\gamma_{i};f)+N_{2}(r,\gamma_{i};\mathcal{L})\big)\\\nonumber&&+\left(\frac{3}{2}-s\right)\ol N_{*}(r,1;F,G)+O(\log r)\\\nonumber&\leq& (2+m)T(r)+\left(\frac{3}{2}-s\right)\big(\ol N_{L}(r,1;F)+\ol N_{L}(r,1;G)\big)+O(\log r)\\\nonumber&\leq&(2+m)T(r)+\frac{\left(\frac{3}{2}-s\right)}{s+1}\big(\ol N(r,0;f)+\ol N(r,0;\mathcal{L})\big)+O(\log r).\eea
When
	\beas&(i)&\;s\geq2,\; n> 2m+4\; or\; when\;\\&(ii)&  s=1,\; n> 2m+5 \;or \;when\;\\&(iii)&  s=0,\;\; n> 2m+10;\eeas
	
  from (\ref{e3.5}), we get a contradiction.
	\\{\it{Case II :}}
Consider $w^{n}+aw^{m}+b=0$ has a multiple root (say $\alpha_{j}$), i.e. the equation has $n-1$ distinct  roots. Then proceeding same as in (\ref{e3.5}) we get \bea\label{e3.6} \frac{n}{2}T(r)&\leq&(2+m)T(r)+\frac{(\frac{3}{2}-s)}{s+1}\big(\ol N(r,0;f)+\ol N(r,0;\mathcal{L})+N(r,\alpha_{j};f)+ N(r,\alpha_{j};\mathcal{L})\big)\\\nonumber&&+O(\log r).\eea
	\par Clearly from (\ref{e3.6}) and for the conditions (i), (ii), (iii) given in {\it{Case I}}, we again arrive at a contradiction.
\par	Therefore from {\it{Case I}} and {\it{Case II}} we have $H\equiv 0$.
In this respect, we  have $F$ and $G$ share 1CM. \par On integration we have \bea\label{e3.7} F=\frac{AG+B}{CG+D},\eea where $A$, $B$, $C$, $D$ are constants such that $AD-BC\not=0$. Thus by {\it{Lemma \ref{l2.3}}} \bea\label{e3.8} T(r,f)=T(r,\mathcal{L})+S(r,\mathcal{L}).\eea
\par As $AD-BC\not=0$, so $A=C=0$ never occur. Thus we consider the following cases:
\\{\it{Case 1 :}} $AC\not=0$ \par In this case \beas F-\frac{A}{C}=\frac{BC-AD}{C(CG+D)}.\eeas
So, $$\ol N(r,\frac{A}{C};F)=\ol N(r,\infty;G).$$Using  the Second Fundamental Theorem and (\ref{e3.8}) we get \beas &&nT(r,f)+O(1)=T(r,F)\\&\leq&\ol N(r,0;F)+\ol N(r,\infty;F)+\ol N(r,\frac{A}{C};F)+S(r,F)\\&\leq&\ol N(r,0;f)+\ol N(r,\infty;f)+\ol N(r,0;af^{m}+b)+\ol N(r,\infty;\mathcal{L})+\ol N(r,0;a\mathcal{L}^{m}+b)+S(r,f)\\&\leq& (2m+1)T(r,f)+O(\log r),\eeas which is a contradiction for $n>2m+4$.
\\{\it{Case 2 :}}\;$AC=0$.
\\{\it{Subcase 2.1 :}}\;$A=0$ and $C\not=0$. \par In this case $B\not=0$ and let us suppose $D\not=0$. Then  $$F=\frac{1}{\gamma G+\delta},$$ where $\gamma=C/B\;,\;\delta=D/B$. If $F$ has no one point, then by the Second Fundamental Theorem we get \beas&&nT(r,f)+O(1)=T(r,F)\\&\leq&\ol N(r,0;F)+\ol N(r,\infty;F)+\ol N(r,1,F)+S(r,F)\\&\leq&\ol N(r,0;f)+\ol N(r,\infty;f)+\ol N(r,0;af^{m}+b)+S(r,f)\\&\leq& (m+1)T(r,f)+O(\log r),\eeas gives a contradiction for $n>2m+4$.

  \par Therefore $1$ is not a exceptional value of $F$. So $\gamma+\delta=1$ and $\gamma\not=0$.\\ So, $$F=\frac{1}{\gamma G+1-\gamma}.$$   Since $D\not=0\implies\delta\not=0$, we have $\gamma\not=1$ and then by the Second Fundamental Theorem and (\ref{e3.8}) we get \beas  &&nT(r,\mathcal{L})+O(1)=T(r,G)\\&\leq&  \ol N(r,0;G)+\ol N(r,\infty;G)+\ol N(r,-(1-\gamma)/\gamma;G)+S(r,G)\\&\leq& \ol N(r,0;\mathcal{L})+\ol N(r,\infty;\mathcal{L})+\ol N(r,0;a\mathcal{L}^{m}+b)+\ol N(r,\infty;f)+\ol N(r,0;af^{m}+b)\\&&+S(r,g)\\&\leq& (2m+1)T(r,\mathcal{L})+O(\log r),\eeas
 which gives a contradiction for $n>2m+4$.
\par Thus our assumption is wrong, so $D=0$ and $\gamma=1$ which implies $FG=1$, impossible by {\it{Lemma \ref{l2.10}}}.
\\{\it{Subcase 2.2 :}}\;$A\not=0$ and $C=0$. \par In this case $D\not=0$ and $$F=\lambda G+\mu,$$ where $\lambda=A/D\;$ and $\mu=B/D$ .\par If $F$ has no one point then proceeding similarly as above we get a contradiction.
 \par Thus $\lambda +\mu=1$ with $\lambda \not=0$ and then $$F=\lambda G+1-\lambda.$$ Then $\ol N(r,-(1-\lambda)/\lambda;G)=\ol N(r,0;F).$
\par Using the Second Fundamental Theorem and (\ref{e3.8}) we get \beas &&nT(r,\mathcal{L})+O(1)=T(r,G)\\&\leq&  \ol N(r,0;G)+\ol N(r,\infty;G)+\ol N(r,-(1-\lambda)/\lambda;G)+S(r,G)\\&\leq&\ol N(r,0;\mathcal{L})+\ol N(r,\infty;\mathcal{L})+\ol N(r,0;a\mathcal{L}^{m}+b)+\ol N(r,0;f)+S(r,g)\\&\leq& (m+2)T(r,\mathcal{L})+O(\log r), \eeas  a contradiction for $n>2m+4$.
\par Therefore $\lambda=1$ and hence \bea\nonumber F &=& G.\\\nonumber i.e.,\;-\frac{f^{n}}{af^{m}+b}&=&-\frac{\mathcal{L}^{n}}{a\mathcal{L}^{m}+b}.\\\label{e3.9}i.e.,\;a\mathcal{L}^{m}(1-h^{n-m})&=&-b(1-h^{n}).\eea \par At first let us assume $h\;(=\frac{\mathcal{L}}{f})$ is a non-constant meromorphic function. \par Then we have, \beas\mathcal{L}^{m}&=&-\frac{b(1-h^{n})}{a(1-h^{n-m})}\\&=&\frac{-b(h-u)(h-u^{2})\ldots(h-u^{n-1})}{a(h-v)(h-v^{2})\ldots(h-v^{n-m-1})},\eeas where gcd(n,m)=1 and $u=\exp\frac{2\pi i}{n}\;,\;v=\exp\frac{2\pi i}{n-m}$.\par Since $\mathcal{L}$ has at most one pole in $\mathbb{C},$ it follows that $h$ has at least $n-m-2$ exceptional values among $\{v,v^{2},\ldots,v^{n-m-1}\}$. Clearly this is a contradiction for $n>2m+4$.\par Hence in that case $h$ is a constant meromorphic function. Now from (\ref{e3.9}) we have $h=1\implies f=\mathcal{L}.$
\end{proof}
\begin{proof}[Proof of Theorem \ref{t1.3}]
Let $F$ and $G$ be defined as in {\it{ Theorem 1.1}}. Clearly $F$, $G$ share $(1,s)$. We consider two cases.\\{\it{Case I}}\;: \; Let $c=0$ and $H\not\equiv 0$. It is  given that $f$ and $\mathcal{L}$ share $(0,0)$. From {\it{Lemma \ref{l2.2}}} we have \beas N(r,\infty;H)&\leq& \ol N_{*}(r,0;f,\mathcal{L})+\ol N_{*}(r,1;F,G)+\ol N(r,0;af^{m}+b\mid\geq 2)+\ol N(r,0;a\mathcal{L}^{m}+b\mid\geq 2)\\&&+\ol N_{0}(r,0;F')+\ol N_{0}(r,0;G')+O(\log r).\eeas Using this and proceeding same as in {\it{Lemma \ref{l2.6}}} we obtain \bea\label{e3.10}\frac{n}{2}T(r)&\leq& 3\ol N(r,0;f)+\sum_{i=1}^{m}\big(N_{2}(r,\gamma_{i};f)+N_{2}(r,\gamma_{i};\mathcal{L})\big)+\left(\frac{3}{2}-s\right)\ol N_{*}(r,1;F,G)\\\nonumber&&+O(\log r),\eea Using {\it{Lemma \ref{l2.12}}} in (\ref{e3.10}) we have \bea\label{e3.11}\;\;\;\;\;\frac{n}{2}T(r)&\leq&\frac{3m}{n-1}T(r)+\sum_{i=1}^{m}\big(N_{2}(r,\gamma_{i};f)+N_{2}(r,\gamma_{i};\mathcal{L})\big)+\left(\frac{3}{2}+\frac{3}{n-1}-s\right)\ol N_{*}(r,1;F,G)\\\nonumber&&+O(\log r).\eea

 First consider $w^{n}+aw^{m}+b=0$ has a multiple root (say $\alpha_{j}$). Then by {\it{Lemma \ref{l2.8}}} and proceeding in the same way as done in {\it{Case II }} of {\it{ Theorem \ref{t1.1}}}, from (\ref{e3.11}) and for \beas&(i)& s\geq2, \; n\geq2m+3\;or\;for\\&(ii)&  s=1,\;\; n\geq2m+4\;or\;for\\&(iii)& s=0,\;\; n\geq 2m+9,\eeas we can get a contradiction.\par Next suppose $w^{n}+aw^{m}+b=0$ has no multiple roots, then dealing in the same way and by the same arguments, again we can get a contradiction. Therefore $H\equiv 0$. 
 
  So we  have $F$ and $G$ share $(1,\infty)$. \par Now by integration we have \bea\nonumber
   F=\frac{AG+B}{CG+D},\eea where $A, B, C, D$ are constant such that $AD-BC\not=0$. \par Again proceeding in the same manner as done in the last part of {\it{ Theorem \ref{t1.1}}} we have \beas F&=&G.\\i.e.,\; \frac{f^{n}}{af^{m}+b}&=&\frac{\mathcal{L}^{n}}{a\mathcal{L}^{m}+b}.\eeas Therefore $f$ and $\mathcal{L}$ share $0$ CM. Then considering $h(=\frac{\mathcal{L}}{f})\not=1$ we have  \beas\mathcal{L}^{m}
   &=&\frac{-b(h-u)(h-u^{2})\ldots(h-u^{n-1})}{a(h-v)(h-v^{2})\ldots(h-v^{n-m-1})},\eeas where gcd(n,m)=1 and $u=\exp\frac{2\pi i}{n}\;,\;v=\exp\frac{2\pi i}{n-m}$.
   
   The possible poles of $\mathcal{L}$ can come from poles of $h$ and $v^{j}\;(j=1,2,\ldots,n-m-1)$ points of $h$. 
   \par Since $\mathcal{L}$ has at most one pole in $\mathbb{C},$ it follows that for $n\geq2m+3$, among these $n-m$ values of $h$, at least  $n-m-1\geq m+2\geq 3$ are exceptional values, a contradiction. 
  Hence in that case $h=1\implies f\equiv \mathcal{L}$. 
 \\{\it{Case II}}\;:
 \;It is given that $c\;(\not\in S)$ is a root of $nz^{n-m}+ma=0$. Therefore $c$ is not a zero of $P(w)$.  Aiso it is given $E_{f}(S,s)=E_{g}(S,s)$ and $E_{f}(\{c\},t)=E_{g}(\{c\},t)$. Let us define $F=-\frac{f^{n}+af^{m}}{b}$ , $G =-\frac{\mathcal{L}^{n}+a\mathcal{L}^{m}}{b}$. Then clearly $F$ and $G$ share $(1,s)$. 
\par According to the hypothesis we know that $P(w)$ has no multiple zeros, i.e. $\alpha_{i}$, $i=1,2,\ldots,n$ are distinct zeros of $P(w)$. Without loss of generality we may assume $c=a_{n-m}$, where $a_{i}\;(i=1,2,\ldots,n-m)$ are given in the statement of {\it{Theorem \ref{t1.3}}}.
\par Assuming $H\not=0$ and from  {\it{Lemma \ref{l2.2}}} we have \beas N(r,\infty;H)&\leq& \ol N(r,0;f)+\ol N(r,0;\mathcal{L})+\sum_{i=1}^{n-m-1}\{\ol N(r,a_{i};f)+\ol N(r,a_{i};\mathcal{L})\}+\ol N_{*}(r,c;f,\mathcal{L})+\ol N_{\odot}(r,0;f')\\&&+\ol N_{*}(r,1;F,G)+\ol N_{\odot}(r,0;\mathcal{L}')+O(\log r),\eeas where 
 $\ol N_{\odot}(r,0;f^{'})$ is the reduced counting function of those zeros of $f'$ which are not zeros of $f(f-a_{1}).\ldots(f-a_{n-m})(F-1)$.
 
 Next by the Second Fundamental Theorem and {\it{ Lemma \ref{l2.5}}}, {\it{ Lemma \ref{l2.1}}} we get, 
 \beas&& (2n-m)T(r)\\ &\leq& \ol N(r,0;f)+\ol N(r,0;\mathcal{L})+\ol N(r,\infty;f)+\ol N(r,\infty;\mathcal{L})+\ol N(r,1;F)+\ol N(r,1;G)\nonumber\\\nonumber&&+ \sum_{i=1}^{n-m}\big(\ol N(r,a_{i};f)+\ol N(r,a_{i};\mathcal{L})\big)-N_{\odot}(r,0;f')-N_{\odot}(r,0;\mathcal{L}')+S(r,f)+S(r,\mathcal{L}).\nonumber\eeas i.e., \beas&&\frac{n}{2}T(r)\\&\leq& 2(\ol N(r,0;f)+\ol N(r,0;\mathcal{L})\big)+\left(\frac{3}{2}-s\right)\ol N_{*}(r,1;F,G)+ (k-1)T(r) +\ol N_{L}(r,c;f)+\ol N_{L}(r,c;\mathcal{L})\\&&+O(\log r).\eeas
 i.e., \be \label{e3.12}\frac{n}{2}T(r)\leq(k+1)T(r)+\left(\frac{3}{2}-s\right)\ol N_{*}(r,1;F,G)+\frac{1}{t+2}( N(r,c;f)+ N(r,c;\mathcal{L}))+O(\log r),\ee where $k=n-m\geq 1$. Next using (\ref{e3.12}) and proceeding in the same way as done in {\it{Case I }} of {\it{Theorem \ref{t1.1}}}, for \beas &(i)& \;s\geq 2,\;t=1\;\;and\;\;n\geq 2k+3, or for \\&(ii)& \;s=1,\;t=0 \;\;and\;\; n\geq 2k+4, or for \\&(iii)& s=0,\;t=0\;\;and\;\; n\geq 2k+7,\eeas we arrive at a contradiction.  
 
   \par Therefore $H\equiv 0$. Then integrating both sides we get, \beas\frac{1}{G-1}=\frac{A}{F-1}+B,\eeas where $A(\not=0),\;B$ are two constants. Again proceeding in the same manner as done in the last part of {\it{ Theorem \ref{t1.2}}} we have \beas  G-1&=&\frac{1}{A}(F-1)\\(\mathcal{L}^{n} + a\mathcal{L}^{m}+b)&=&\frac{1}{A}(f^{n} + af^{m}+b)\\\mathcal{L}^{n} + a\mathcal{L}^{m}+b-\frac{b}{A}&=&\frac{1}{A}(f^{n} + af^{m})\eeas and dealing in the same way as in the rest part of {\it{Theorem 1.2}} we will get $f=\mathcal{L}$.
 \end{proof}

\end{document}